\documentclass[12pt, a4paper]{amsart}
\usepackage{mathptmx}
\usepackage[top=25mm, bottom=30mm,left=25mm,right=25mm]{geometry}
\usepackage{amsmath, amssymb,amsthm}
\usepackage[colorlinks=true,citecolor=blue]{hyperref}
\usepackage[mathlines]{lineno}

\newtheorem{theorem}{Theorem}[section]
\newtheorem{lemma}[theorem]{Lemma}
\newtheorem{ques}{Question}

\newcounter{MainTheoremCounter}

\newtheorem{Maintheorem}[MainTheoremCounter]{Theorem}

\theoremstyle{definition}

\numberwithin{equation}{section}
\newtheorem*{ac}{Acknowledgements}
\newtheorem*{orga}{Organization of the paper}

\begin{document}

\title{Zero-Threshold Discrepancies for Multiple Correlation Sequences}
\author[K.~Ouyang]{Kangbo Ouyang}
\email{oy19981231@mail.ustc.edu.cn}
\author[J.~Qiu]{Jiahao Qiu}
\email{qiujh@mail.ustc.edu.cn}
\author[X.~Ye]{Xiangdong Ye}
\email{yexd@ustc.edu.cn}
\address[K.~Ouyang, J.~Qiu, X.~Ye]
{School of Mathematical Sciences\\
  University of Science and Technology of China\\
  Hefei, Anhui 230026, PR China}

\date{\today}

\keywords{Pro-nilfactors, zero-threshold lifting, multiple correlation sequences}
\subjclass[2020]{Primary 37A45; Secondary 37A30, 05D10}

\begin{abstract}
  We compare polynomial multiple correlation sequences with their counterparts
  on the measure-theoretic pro-nilfactor $\mathcal Z_\infty$.
  Let $(X,\mu,T)$ be an ergodic system, let $f_1,\ldots,f_d$ be non-negative
  bounded functions, and let $p_1,\ldots,p_d\in\mathbb Z[t]$ be distinct
  polynomials with zero constant term. We prove that the set
  \[
    D_0:=\{n\in\mathbb Z:\beta(n)>0\}
    \setminus
    \{n\in\mathbb Z:\alpha(n)>0\}
  \]
  is not piecewise syndetic, where
  \[
    \alpha(n)
    =
    \int_X
    T^{p_1(n)}f_1\cdots T^{p_d(n)}f_d
    \,d\mu
    \quad\text{and}\quad
    \beta(n)
    =
    \int_X
    T^{p_1(n)}\mathbb E(f_1\mid\mathcal Z_\infty)
    \cdots
    T^{p_d(n)}\mathbb E(f_d\mid\mathcal Z_\infty)
    \,d\mu.
  \]
  Moreover, we construct an ergodic system $(X,\mu,T)$, non-negative bounded
  functions $f_1$, $f_2$, and polynomials $p_1(n)=0$, $p_2(n)=n$ such
  that $D_0$ has positive upper density. This provides a negative answer to a
  question raised by Glasscock, Koutsogiannis, Le, Moreira, Richter, and
  Robertson.
\end{abstract}

\maketitle

\section{Introduction}

We recall the background on characteristic factors and correlation sequences,
and state our main results.

\subsection{Background}

\subsubsection{Characteristic factors}

Szemer\'edi's theorem~\cite{Sze75} states that every subset of the positive
integers with positive upper density contains arithmetic progressions of
arbitrary finite length. Furstenberg~\cite{F77} gave an ergodic proof of
Szemer\'edi's theorem. Using the \emph{correspondence principle}
in~\cite{F81}, Furstenberg reduced the problem to recurrence in a
measure-preserving system: for every measure-preserving transformation $T$ of
a probability space $(X,\mu)$ and every measurable set $A$ with $\mu(A)>0$ one
has
\[
  \liminf_{N\to\infty}\frac1N\sum_{n=1}^N
  \mu\bigl(A\cap T^{-n}A\cap\cdots\cap T^{-dn}A\bigr)>0.
\]
This motivates the study of the multiple ergodic averages
\begin{equation}\label{mea}
  A_N(g_1,\ldots,g_d)(x)
  :=
  \frac1N\sum_{n=1}^N g_1(T^nx)g_2(T^{2n}x)\cdots g_d(T^{dn}x),
\end{equation}
where $g_1,\ldots,g_d$ are bounded measurable functions.

In the study of the averages \eqref{mea}, the notion of a characteristic
factor plays a central role. To be precise, let $(X,\mathcal X,\mu,T)$ be a
measure-preserving system and let $(Y,\mathcal Y,\nu,T)$ be a factor of $X$.
For $d\ge1$, the factor $\mathcal Y$ is called \emph{characteristic} for the
pattern $n,2n,\ldots,dn$ if for all $g_1,\ldots,g_d\in L^\infty(\mu)$,
\begin{equation*}
  \frac1N\sum_{n=1}^{N}
  \left(
  \prod_{i=1}^{d}T^{in}g_i
  -
  \prod_{i=1}^{d}T^{in}\mathbb E(g_i\mid\mathcal Y)
  \right)
  \longrightarrow0
  \quad\text{in }L^2(\mu).
\end{equation*}
Once a characteristic factor is known, the limit behavior of the averages is
determined on that factor alone. For ergodic systems, the characteristic
factors for linear multiple averages were obtained by
Host--Kra~\cite{HK05} and Ziegler~\cite{TZ07}. We write $\mathcal Z_k$ for the
maximal $k$-step pro-nilfactor. Each $\mathcal Z_k$ is an inverse limit of
$k$-step nilsystems. For polynomial iterates, the corresponding theory is due
to Host--Kra~\cite{HK05B} and Leibman~\cite{Le05C}; see also~\cite{HK18}.

\subsubsection{Multiple correlation sequences}

A multiple correlation sequence is a sequence of the form
\[
  n\mapsto
  \int_X h_0\,T^{p_1(n)}h_1\cdots T^{p_d(n)}h_d\,d\mu,
  \qquad n\in\mathbb Z,
\]
where $(X,\mu,T)$ is a measure-preserving system,
$h_0,\ldots,h_d\in L^\infty(\mu)$, and $p_1,\ldots,p_d\in\mathbb Z[t]$. For
$\varepsilon\ge0$, a sequence $e\colon\mathbb Z\to\mathbb C$ is
\emph{$\varepsilon$-null} if
\begin{equation}\label{eq:epsnull}
  \limsup_{N-M\to\infty}
  \frac1{N-M}\sum_{n=M}^{N-1}|e(n)|^2
  \le\varepsilon,
\end{equation}
and a \emph{null-sequence} is a $0$-null sequence.

For arithmetic-progression correlations,
Bergelson--Host--Kra~\cite[Theorem 1.9]{BHK05} proved that for an ergodic
system and a real-valued function $f\in L^\infty(\mu)$,
\begin{equation}\label{BHK-eq}
  \int_X f\,T^nf\cdots T^{dn}f\,d\mu
  =
  N(n)+e(n),
\end{equation}
where $N$ is a $d$-step nilsequence and $e$ is a null-sequence.
Leibman~\cite{Le10} extended this to polynomial iterates, and in~\cite{Le15}
removed the ergodicity assumption and treated correlations indexed by
$\mathbb Z^d$.

For a positive integer $r$, write $[r]:=\{1,\ldots,r\}$.
In the case of several commuting transformations,
Frantzikinakis~\cite[Theorem 1.2]{FN15} established that for every
$s,k\in\mathbb N$ and any polynomials $p_{i,j}\in\mathbb Z[t]$, $i\in[s]$,
$j\in[k]$, there exists $\ell\in\mathbb N$ such that for every $s$-tuple of
commuting measure-preserving transformations $T_1,\ldots,T_s$ of $(X,\mu)$,
all $f_1,\ldots,f_k\in L^\infty(\mu)$, and every $\varepsilon>0$ one has the
decomposition
\begin{equation}\label{eq:approx}
  \int_X
  \Bigl(\prod_{i=1}^{s}T_i^{p_{i,1}(n)}\Bigr)f_1
  \cdots
  \Bigl(\prod_{i=1}^{s}T_i^{p_{i,k}(n)}\Bigr)f_k
  \,d\mu
  =
  N_\varepsilon(n)+e_\varepsilon(n),
  \qquad n\in\mathbb N,
\end{equation}
where $N_\varepsilon$ is an $\ell$-step nilsequence and $e_\varepsilon$ is an
$\varepsilon$-null sequence on $\mathbb N$, meaning that
\eqref{eq:epsnull} holds with $1\le M<N$.

Later, exact decompositions were obtained for linear iterates of commuting
transformations by Leng~\cite{Len25}, and for arbitrary polynomial iterates of
nilpotent group actions by Andreevski--Hern\'andez~\cite{AH26}.
In Leng's result, the null condition is defined using the symmetric
intervals $[-N,N]$.
See also~\cite{DMKS24,FK25} and the references therein.

\subsubsection{Positivity on the pro-nilfactor}

In \eqref{BHK-eq} the remainder is a null-sequence, which is small in the
mean over long intervals; this controls an average, not the individual
values. In this paper we focus on the multiple correlation
sequences of a single action $T$ and compare them with their counterparts on
the pro-nilfactor $\mathcal Z_\infty$.

Let $(X,\mu,T)$ be an ergodic system. Put
$\mathcal P_0:=\{p\in\mathbb Z[t]:p(0)=0\}$ and fix distinct polynomials
$p_1,\ldots,p_d\in\mathcal P_0$. In the discussion, all functions are
real-valued and non-negative. For $f_1,\ldots,f_d\in L^\infty(\mu)$ define
\begin{equation}\label{defn-a}
  \alpha(n)
  :=
  \int_X T^{p_1(n)}f_1\cdots T^{p_d(n)}f_d\,d\mu
\end{equation}
and
\begin{equation}\label{defn-b}
  \beta(n)
  :=
  \int_X
  T^{p_1(n)}\mathbb E(f_1\mid\mathcal Z_\infty)
  \cdots
  T^{p_d(n)}\mathbb E(f_d\mid\mathcal Z_\infty)
  \,d\mu.
\end{equation}

For $f\ge0$ one has
$\{f>0\}\subseteq\{\mathbb E(f\mid\mathcal Z_\infty)>0\}$ up to a null set,
and $\mathcal Z_\infty$ is $T$-invariant; hence $\alpha(n)>0$ implies
$\beta(n)>0$ for every $n\in\mathbb Z$. The characteristic factor property
implies that $\alpha-\beta$ is Banach-null, that is,
\[
  \lim_{N-M\to\infty}
  \frac1{N-M}\sum_{n=M}^{N-1}|\alpha(n)-\beta(n)|
  =
  0.
\]
Consequently, for $0<\varepsilon_1<\varepsilon_2$, Glasscock, Koutsogiannis,
Le, Moreira, Richter, and Robertson~\cite[Theorem 5.8]{GKLMMRR} prove the set
\[
  \{n\in\mathbb Z:\beta(n)>\varepsilon_2\}
  \setminus
  \{n\in\mathbb Z:\alpha(n)>\varepsilon_1\}
\]
has zero upper Banach density.

When $\varepsilon_1=\varepsilon_2=0$, however, this $L^1$-approximation does
not control the set of times for which $\beta(n)$ is strictly positive while
$\alpha(n)$ vanishes. The following question arises in the study of polynomial
return times~\cite[Question 5.9]{GKLMMRR}.

\begin{ques}\label{ques-1}
  Let $(X,\mu,T)$ be an ergodic system, and let $p_1,\ldots,p_d\in\mathbb Z[t]$
  satisfy $p_i(0)=0$ with $p_i-p_j$ non-constant for $i\ne j$. Let
  $f_1,\ldots,f_d\in L^\infty(\mu)$ be non-negative functions, and let
  $\alpha(n)$ and $\beta(n)$ be defined by \eqref{defn-a} and \eqref{defn-b}.
  Is it true that
  \[
    d^*\bigl(
    \{n\in\mathbb Z:\beta(n)>0\}
    \setminus
    \{n\in\mathbb Z:\alpha(n)>0\}
    \bigr)=0?
  \]
\end{ques}

\subsection{Main results}

Write
\[
  D_0
  :=
  \{n\in\mathbb Z:\beta(n)>0\}\setminus\{n\in\mathbb Z:\alpha(n)>0\}
\]
for the \emph{zero-threshold discrepancy set}. The first of our main results
answers Question~\ref{ques-1} negatively.

\begin{Maintheorem}
\label{thm:positive-density-counterexample}
  There exist an ergodic system $(X,\mu,T)$ and measurable sets
  $A,B\subseteq X$ of positive measure such that, with the polynomials $p_1(n)=0$,
  $p_2(n)=n$ and the functions $f_1=\mathbf 1_A$, $f_2=\mathbf 1_B$, one has
  \[
    \overline d\bigl(
    \{n\in\mathbb Z:\beta(n)>0\}
    \setminus
    \{n\in\mathbb Z:\alpha(n)>0\}
    \bigr)>0.
  \]
\end{Maintheorem}

In particular, Question~\ref{ques-1} has a negative answer.
Theorem~\ref{thm:positive-density-counterexample} says nothing about the
arrangement of the times in $D_0$. Our second result shows that $D_0$ is never
piecewise syndetic.

\begin{Maintheorem}
\label{thm:non-piecewise-syndetic}
  Let $(X,\mu,T)$ be an ergodic system and let $\mathcal Z_\infty$ be its
  $\infty$-step pro-nilfactor. Let $p_1,\ldots,p_d\in\mathcal P_0$ be distinct
  polynomials, and let $f_1,\ldots,f_d\in L^\infty(\mu)$ be non-negative
  functions. Then the set
  \[
    \{n\in\mathbb Z:\beta(n)>0\}\setminus\{n\in\mathbb Z:\alpha(n)>0\}
  \]
  is not piecewise syndetic.
\end{Maintheorem}

In topological dynamics, the analogous statement
is~\cite[Theorem A]{GKLMMRR}: for a minimal system, the extra return times
that appear on the maximal $\infty$-step pro-nilfactor form a non-piecewise
syndetic set. Theorem~\ref{thm:non-piecewise-syndetic} is the corresponding
statement for positive correlations in an ergodic system.

\begin{orga}
  Section~2 collects notation and known facts about pro-nilfactors and
  polynomial correlation sequences. Section~3 constructs the system of
  Theorem~\ref{thm:positive-density-counterexample}, and Section~4 proves
  Theorem~\ref{thm:non-piecewise-syndetic}.
\end{orga}

\begin{ac}
  The authors would like to thank Professors Wen Huang and Song Shao for
  helpful discussions. This research was supported by the National Key
  Research and Development Program of China (Nos.~2024YFA1013601 and
  2024YFA1013600), the University of Science and Technology of China Research
  Funds of the Double First-Class Initiative (YD0010002009), and the National
  Natural Science Foundation of China (Nos.~12031019, 12401243, 12426201, and
  12471188).
\end{ac}

\section{Preliminaries}

We collect standard notation and background results used throughout the paper.
Let $\mathbb N$ and $\mathbb Z$ denote the sets of positive integers and
integers, respectively. For any finite set $F$, $|F|$ denotes its cardinality.

\subsection{Subsets of integers}

Let $E\subseteq\mathbb Z$. The \emph{upper density} and \emph{upper Banach
density} of $E$ are defined, respectively, as
\[
  \overline d(E)
  :=
  \limsup_{N\to\infty}
  \frac{|E\cap[-N,N]|}{2N+1}
  \qquad\text{and}\qquad
  d^*(E)
  :=
  \limsup_{N-M\to\infty}
  \frac{|E\cap[M,N-1]|}{N-M}.
\]
By definition, $d^*(E)\ge \overline d(E)$.

For a subset $A\subseteq\mathbb Z$ and an integer $t\in\mathbb Z$, we denote
the shifted set by
\[
  A-t:=\{n\in\mathbb Z:n+t\in A\}.
\]
A family $\mathcal F$ of subsets of $\mathbb Z$ is \emph{upward-closed} if
$F\in\mathcal F$ and $F\subseteq F'\subseteq\mathbb Z$ imply $F'\in\mathcal F$.
Given a subset $A\subseteq\mathbb Z$ and a family $\mathcal F$, we define
\[
  A-\mathcal F
  :=
  \{n\in\mathbb Z:A-n\in\mathcal F\}.
\]

We recall the following standard notions. A set $E\subseteq\mathbb Z$ is
\emph{thick} if it contains arbitrarily long intervals of integers. It is
\emph{syndetic} if there exists an integer $L\ge 1$ such that every interval
of length $L$ intersects $E$. A set is \emph{piecewise syndetic} if it is the
intersection of a thick set and a syndetic set.

We use the following elementary fact about syndetic sets.

\begin{lemma}\label{lem:syndetic-minus-null}
  If $S\subseteq\mathbb Z$ is syndetic and $E\subseteq\mathbb Z$ has zero
  upper Banach density, then the set difference $S\setminus E$ is syndetic.
\end{lemma}

\begin{proof}
  Since $S$ is syndetic, there exists an integer $L\ge 1$ such that every
  interval of length $L$ intersects $S$. Suppose for contradiction that
  $S\setminus E$ is not syndetic. Then there exists a sequence of intervals
  $I_m\subseteq\mathbb Z$ with lengths $|I_m|\to\infty$ such that
  $I_m\cap(S\setminus E)=\emptyset$. This implies that
  \[
    I_m\cap S\subseteq I_m\cap E.
  \]
  Partitioning $I_m$ into contiguous, disjoint subintervals of length $L$, and
  discarding at most $L-1$ residual points, we have
  \[
    |I_m\cap S|
    \ge
    \left\lfloor\frac{|I_m|}{L}\right\rfloor
    \ge
    \frac{|I_m|}{L}-1.
  \]
  Therefore,
  \[
    \frac{|I_m\cap E|}{|I_m|}
    \ge
    \frac{|I_m\cap S|}{|I_m|}
    \ge
    \frac{1}{L}-\frac{1}{|I_m|}.
  \]
  Taking the limit supremum as $m\to\infty$ gives
  $d^*(E)\ge \frac{1}{L}>0$, contradicting the assumption that $d^*(E)=0$.
  Therefore, $S\setminus E$ is syndetic.
\end{proof}

\subsection{Measure-preserving systems}

A \emph{measure-preserving system} $(X,\mathcal X,\mu,T)$ consists of a
standard probability space $(X,\mathcal X,\mu)$ and an invertible
measure-preserving transformation $T:X\to X$. We typically omit the
$\sigma$-algebra $\mathcal X$ and write $(X,\mu,T)$.

The invariant $\sigma$-algebra is denoted by
$\mathcal I(T) = \{A\in\mathcal X:T^{-1}A=A\}$. The system is \emph{ergodic}
if $\mathcal I(T)$ is trivial (meaning every invariant set has measure $0$ or
$1$), and it is \emph{weakly mixing} if the product system
$(X\times X,\mu\times\mu,T\times T)$ is ergodic.

A \emph{factor} of $(X,\mu,T)$ is identified with a $T$-invariant
sub-$\sigma$-algebra $\mathcal B$ of $\mathcal X$. For $f\in L^1(\mu)$, the
conditional expectation of $f$ with respect to $\mathcal B$ is the unique
(up to equality $\mu$-almost everywhere) $\mathcal B$-measurable integrable
function $\mathbb E(f\mid\mathcal B)$ satisfying
\[
  \int_B f\,d\mu
  =
  \int_B \mathbb E(f\mid\mathcal B)\,d\mu,
  \qquad B\in\mathcal B.
\]
We identify factors with the corresponding invariant sub-$\sigma$-algebras and
factor spaces when no ambiguity arises.

\subsection{Nilsequences and Banach-null sequences}

Let $G$ be a group. For $g,h\in G$, we denote the commutator by
$[g,h]:=g^{-1}h^{-1}gh$. The lower central series of $G$ is defined
inductively by $G_1:=G$ and $G_{j+1}:=[G_j,G]$ for $j\ge 1$. The group $G$ is
$k$-step nilpotent if $G_{k+1}=\{e\}$.

Let $G$ be a $k$-step nilpotent Lie group, and let $\Gamma\le G$ be a discrete
cocompact subgroup. The compact homogeneous space $G/\Gamma$ is called a
$k$-step \emph{nilmanifold}. If $a\in G$ and $R_a$ denotes the left
translation defined by $R_a(g\Gamma):=ag\Gamma$, then $(G/\Gamma,R_a)$ is
called a $k$-step \emph{nilsystem}. Viewed as a measure-preserving system, a
nilsystem is canonically equipped with its Haar probability measure (see
Parry~\cite{Par69} for the ergodic properties of such transformations). A
$k$-step \emph{pro-nilsystem} is an inverse limit of $k$-step nilsystems.

A bounded sequence $\psi:\mathbb Z\to\mathbb C$ is called a \emph{nilsequence}
if there exist a pro-nilsystem $(Y,R)$, a point $y_0\in Y$, and a continuous
function $F\in C(Y)$ such that
\[
  \psi(n)=F(R^ny_0),
  \qquad n\in\mathbb Z.
\]
If $Y$ can be chosen to be a $k$-step pro-nilsystem, $\psi$ is referred to as
a $k$-step nilsequence.

A bounded sequence $\lambda:\mathbb Z\to\mathbb C$ is called
\emph{Banach-null} if
\[
  \lim_{N-M\to\infty}
  \frac{1}{N-M}\sum_{n=M}^{N-1}|\lambda(n)|
  =
  0.
\]
Equivalently, for every $\varepsilon>0$,
\[
  d^*\bigl(\{n\in\mathbb Z:|\lambda(n)|>\varepsilon\}\bigr)=0.
\]

We also need the following lemma.

\begin{lemma}\label{lem:nilsequence-banach-null-algebra}
  Let $r\ge1$. Suppose that $a_j(n)=\psi_j(n)+\lambda_j(n)$ for each
  $j\in\{1,\ldots,r\}$, where $\psi_j$ is a nilsequence and $\lambda_j$ is
  Banach-null. Assume also that all sequences involved are bounded. Then there
  exist a nilsequence $\psi$ and a Banach-null sequence $\lambda$ such that
  \[
    \prod_{j=1}^{r}a_j(n)=\psi(n)+\lambda(n).
  \]
\end{lemma}

\begin{proof}
  Set $\psi(n):=\prod_{j=1}^{r}\psi_j(n)$. Since nilsequences are closed under
  finite products, $\psi$ is a nilsequence.

  Expanding the product, we obtain
  \[
    \prod_{j=1}^{r}a_j(n)
    =
    \prod_{j=1}^{r}\psi_j(n)
    +
    \lambda(n),
  \]
  where $\lambda$ is a finite sum of terms of the form $u(n)\lambda_j(n)$,
  with $j\in\{1,\ldots,r\}$ and $u$ a bounded sequence. Indeed, every term
  other than $\prod_{j=1}^r\psi_j(n)$ contains at least one Banach-null factor
  $\lambda_j(n)$, and all remaining factors are bounded.

  It remains to show that each such term is Banach-null. Since $u$ is bounded,
  say $\|u\|_\infty\le C$, we have
  \[
    \frac{1}{N-M}\sum_{n=M}^{N-1}|u(n)\lambda_j(n)|
    \le
    C\,\frac{1}{N-M}\sum_{n=M}^{N-1}|\lambda_j(n)|.
  \]
  The right-hand side tends to $0$ as $N-M\to\infty$. Hence $u\lambda_j$ is
  Banach-null. Since a finite sum of Banach-null sequences is Banach-null,
  $\lambda$ is Banach-null.
\end{proof}

\begin{lemma}\label{lem:nilsequence-positive-set-syndetic}
  Let $\psi:\mathbb Z\to\mathbb R$ be a nilsequence. If
  $a:=\sup_{n\in\mathbb Z}\psi(n)>0$, then for every $0<c<a$, the set
  $\{n\in\mathbb Z:\psi(n)>c\}$ is syndetic.
\end{lemma}

\begin{proof}
  By the definition of a nilsequence, there exist a pro-nilsystem $(Y,R)$, a
  point $y_0\in Y$, and a continuous function $\Phi\in C(Y,\mathbb R)$ such
  that
  \[
    \psi(n)=\Phi(R^ny_0),
    \qquad n\in\mathbb Z.
  \]
  Choose $n_0\in\mathbb Z$ satisfying $\psi(n_0)>c$, and set
  \[
    Z:=\overline{\{R^ny_0:n\in\mathbb Z\}}.
  \]
  The orbit closure of any point in a pro-nilsystem is minimal, so $(Z,R)$ is
  minimal.

  Let $V:=\{y\in Z:\Phi(y)>c\}$. Then $V$ is an open subset of $Z$, and it is
  non-empty since $R^{n_0}y_0\in V$. By minimality, the family
  $\{R^{-m}V:m\in\mathbb Z\}$ forms an open cover of $Z$. Since $Z$ is
  compact, we can extract a finite subcover, meaning there exist
  $m_1,\ldots,m_s\in\mathbb Z$ such that
  \[
    Z=\bigcup_{i=1}^s R^{-m_i}V.
  \]
  Let $M:=\max_{1\le i\le s}|m_i|$. For every $n\in\mathbb Z$, the point
  $R^ny_0\in Z$ belongs to $R^{-m_i}V$ for some $1\le i\le s$. Hence,
  \[
    R^{n+m_i}y_0\in V.
  \]
  Since $|m_i|\le M$, this implies that
  \[
    [n-M,n+M]\cap\{m\in\mathbb Z:R^my_0\in V\}\ne\emptyset
  \]
  for every $n\in\mathbb Z$. Thus, the return-time set
  $\{m\in\mathbb Z:R^my_0\in V\}$ has bounded gaps and is therefore syndetic.
  This set is precisely $\{n\in\mathbb Z:\psi(n)>c\}$.
\end{proof}

\subsection{Pro-nilfactors and polynomial correlation sequences}

We recall pro-nilfactors and standard structural results for
polynomial multiple correlation sequences. The Host--Kra seminorms are
ergodic-theoretic analogues of the uniformity norms introduced by
Gowers~\cite{Gow01} in additive combinatorics.

Let $(X,\mu,T)$ be an ergodic system. For each $k\ge 0$, the $k$-step
pro-nilfactor $\mathcal Z_k$ is characterized by the seminorm
$\|\cdot\|_{U^{k+1}}$: for every $f\in L^\infty(\mu)$,
\[
  \mathbb E(f\mid\mathcal Z_k)=0
  \quad\Longleftrightarrow\quad
  \|f\|_{U^{k+1}}=0.
\]
In particular, $\mathcal Z_0$ is the invariant factor, which is trivial when
the system is ergodic. By the Host--Kra--Ziegler structure
theorem~\cite{HK05,TZ07}, each $\mathcal Z_k$, $k\ge 1$, is an inverse limit
of $k$-step nilsystems. We define the $\infty$-step pro-nilfactor as
$\mathcal Z_\infty := \bigvee_{k\ge1}\mathcal Z_k$.

We record the following standard facts; see
\cite{BHK05,HK05B,Le05C,Le10,Le15}.

\begin{theorem}[Standard facts on polynomial correlation sequences]
\label{thm:structural-facts-polynomial-correlations}
  Let $(X,\mu,T)$ be an ergodic system. Let $q_1,\ldots,q_\ell\in\mathcal P_0$
  be distinct polynomials, and let $f_1,\ldots,f_\ell\in L^\infty(\mu)$. Then
  the following assertions hold.
  \begin{enumerate}
    \item \emph{Banach-null characteristic approximation.} The difference
    sequence
    \[
      n\mapsto
      \int_X\prod_{i=1}^{\ell}T^{q_i(n)}f_i\,d\mu
      -
      \int_X
      \prod_{i=1}^{\ell}T^{q_i(n)}\mathbb E(f_i\mid\mathcal Z_\infty)\,d\mu
    \]
    is Banach-null; see \cite[proof of Theorem~5.8]{GKLMMRR}.

    \item \emph{Nilsequence--null decomposition.} The polynomial multiple
    correlation sequence
    \[
      a(n)
      :=
      \int_X\prod_{i=1}^{\ell}T^{q_i(n)}f_i\,d\mu,
      \qquad n\in\mathbb Z,
    \]
    admits a decomposition $a(n)=\psi(n)+\lambda(n)$, where $\psi$ is a
    nilsequence and $\lambda$ is Banach-null; see
    \cite[Theorem~3.1]{Le10}.
  \end{enumerate}
\end{theorem}

\subsection{The pro-nilfactor of the rotation--Bernoulli extension}

We determine the $\infty$-step pro-nilfactor of the product of an irrational
circle rotation and a Bernoulli shift, the system underlying the
counterexample in Section~\ref{sec:counterexample}.

\begin{lemma}\label{lem:rotation-bernoulli-pronilfactor}
  Let $R_\theta$ be an irrational rotation on $\mathbb T$, and let
  $(\Omega,\nu,S)$ be the two-sided Bernoulli shift. For the product system
  \[
    (X,\mu,T)
    :=
    (\mathbb T\times\Omega,\ m\times\nu,\ R_\theta\times S),
  \]
  one has
  $\mathcal Z_\infty = \mathcal B_{\mathbb T}\otimes\{\emptyset,\Omega\}$
  modulo null sets.
\end{lemma}

\begin{proof}
  The extension $X\to\mathbb T$ given by the coordinate projection is
  relatively weakly mixing. Indeed, its relative product over $\mathbb T$ is
  isomorphic to the system $\mathbb T\times\Omega\times\Omega$ equipped with
  the transformation $R_\theta\times S\times S$. Since the Bernoulli shift $S$
  is weakly mixing, so is $S\times S$. The product of an ergodic
  system with a weakly mixing system is ergodic, so $R_\theta\times S\times S$
  is ergodic. Thus, the extension $X\to\mathbb T$ is relatively weakly mixing.

  The circle rotation factor
  $\mathcal B_{\mathbb T}\otimes\{\emptyset,\Omega\}$ is contained in
  $\mathcal Z_\infty$ because an irrational rotation is a $1$-step nilsystem.

  Since $\mathcal Z_\infty$ is an inverse limit of finite-step nilsystems, it
  is a distal factor. By the relative Furstenberg--Zimmer structure
  theorem~\cite{F77}, a relatively weakly mixing extension admits no
  non-trivial relatively distal intermediate factor. Therefore, the distal
  intermediate factor $\mathcal Z_\infty$ between the rotation factor and $X$
  coincides with the rotation factor modulo null sets.
\end{proof}

\section{Proof of Theorem~\ref{thm:positive-density-counterexample}}
\label{sec:counterexample}

In this section, we prove Theorem~\ref{thm:positive-density-counterexample}.
We construct an ergodic system and two measurable sets of positive measure
for which the zero-threshold discrepancy set has positive upper density.

\subsection{A circle covariogram with a large zero set}

Let $\mathbb T=\mathbb R/\mathbb Z$, and let $m$ denote the Lebesgue measure
on $\mathbb T$. For a subset $A\subseteq\mathbb T$ and $\rho>0$, define the
open $\rho$-neighborhood of $A$ by
\[
  A_\rho:=\{x\in\mathbb T:\operatorname{dist}_{\mathbb T}(x,A)<\rho\}.
\]

\begin{lemma}\label{lem:open-set}
  There exists an open set $O\subseteq\mathbb T$ such that $O-O$ is dense and
  open in $\mathbb T$, but $m(O-O)<1$. Consequently, if $U:=O-O$ and
  $C:=\mathbb T\setminus U$, then $C$ is closed and has positive Lebesgue
  measure.
\end{lemma}

\begin{proof}
  Let $H\subseteq\mathbb T$ be the image of the classical middle-thirds Cantor
  set under the quotient map $[0,1]\to\mathbb T$. Then $H$ is compact,
  $m(H)=0$, and $H-H=\mathbb T$: the middle-thirds Cantor set
  $K\subset[0,1]$ satisfies $K-K=[-1,1]$, whose image modulo $1$ is
  $\mathbb T$. Choose a countable dense subset
  $D=\{a_1,a_2,\ldots\}\subseteq H$. For each $i$, the sets $a_i-H$ and
  $H-a_i$ are compact null subsets of $\mathbb T$. Choose positive numbers
  $\varepsilon_i$ with $\sum_{i=1}^\infty \varepsilon_i<1$.

  The Lebesgue measure of the $\rho$-neighborhood of any compact null set
  tends to $0$ as $\rho\to0$. Thus, we may choose $r_i>0$ inductively such
  that
  \[
    m((a_i-H)_{2r_i})+m((H-a_i)_{2r_i})<\varepsilon_i,
  \]
  and $r_i<\min\{2^{-i}, r_{i-1}\}$ for $i\ge2$. In particular, $r_i\to0$.

  Define the open set
  \[
    O:=\bigcup_{i=1}^{\infty}B(a_i,r_i).
  \]
  Since $D\subseteq O$ and $D$ is dense in $H$, the difference set $O-O$
  contains $D-D$, which is dense in $H-H=\mathbb T$. Hence, $O-O$ is dense. It
  is also open, since
  \[
    O-O=\bigcup_{y\in O}(O-y).
  \]
  To estimate the measure of $O-O$, let $x\in B(a_i,r_i)$ and
  $y\in B(a_j,r_j)$. If $i\le j$, then $r_j\le r_i$, implying
  $x-y\in(a_i-H)_{2r_i}$. If $j<i$, then similarly
  $x-y\in(H-a_j)_{2r_j}$. Therefore,
  \[
    O-O
    \subseteq
    \bigcup_{i=1}^{\infty}(a_i-H)_{2r_i}
    \cup
    \bigcup_{i=1}^{\infty}(H-a_i)_{2r_i}.
  \]
  By countable subadditivity,
  \begin{align*}
    m(O-O)
    &\le
    \sum_{i=1}^{\infty}
    \Bigl(m((a_i-H)_{2r_i})+m((H-a_i)_{2r_i})\Bigr) \\
    &<
    \sum_{i=1}^{\infty}\varepsilon_i<1. \qedhere
  \end{align*}
\end{proof}

Fix such an open set $O$ for the remainder of this section. Let
\[
  U:=O-O,
  \qquad
  C:=\mathbb T\setminus U,
  \qquad
  c:=m(C)>0.
\]
Define the covariogram $h:\mathbb T\to\mathbb R$ by
\[
  h(t):=m(O\cap(O-t)).
\]
The function $h$ is continuous by the continuity of translations in
$L^1(\mathbb T)$. Moreover,
\[
  h(t)>0
  \quad\iff\quad
  O\cap(O-t)\ne\emptyset
  \quad\iff\quad
  t\in U.
\]
Indeed, if $O\cap(O-t)$ is non-empty, it is a non-empty open subset of
$\mathbb T$ and thus has positive Lebesgue measure. Consequently, $h$ vanishes
identically on $C$. Note that $U$ is symmetric and $h$ is even, meaning
$h(-t)=h(t)$ for all $t\in\mathbb T$.

\subsection{A positive-density set of small covariogram values}

We record an elementary averaging estimate.

\begin{lemma}\label{lem:interval-estimate-unified}
  Let $I\subseteq\mathbb T$ be a non-empty open arc. Then, for every
  measurable set $W\subseteq\mathbb T$ and every integer $n\ge1$,
  \[
    \int_I \mathbf 1_W(n\theta)\,dm(\theta)
    \ge
    m(I)m(W)-\frac{2}{n}.
  \]
\end{lemma}

\begin{proof}
  Choose an interval $(a,b)\subseteq\mathbb R$ that projects bijectively onto
  $I$, so $b-a=m(I)$. We regard $\mathbf 1_W$ as a $1$-periodic function on
  $\mathbb R$. Using the change of variables $t=n\theta$, we have
  \[
    \int_I \mathbf 1_W(n\theta)\,dm(\theta)
    =
    \frac1n\int_{na}^{nb}\mathbf 1_W(t)\,dt.
  \]
  Let $J=(na,nb)$, and let $J'$ be the union of all unit intervals with
  integer endpoints that are fully contained in $J$. Then $J'\subseteq J$ and
  $|J\setminus J'|\le2$. Since $\mathbf 1_W$ is $1$-periodic, its integral over
  each such unit interval is precisely $m(W)$. Thus,
  \begin{align*}
    \int_{na}^{nb}\mathbf 1_W(t)\,dt
    &\ge
    \int_{J'}\mathbf 1_W(t)\,dt \\
    &=
    |J'|m(W) \\
    &\ge
    (nm(I)-2)m(W) \\
    &\ge
    nm(I)m(W)-2.
  \end{align*}
  Dividing by $n$, we obtain the desired estimate.
\end{proof}

\begin{lemma}\label{lem:unified-E}
  Let $U$ and $h$ be as defined before. There exist an irrational
  $\theta\in\mathbb T$ and a set $E\subseteq\mathbb Z$ such that
  $n\theta\in U$ for every $n\in E$, $\sum_{n\in E}h(n\theta)<\infty$, and
  $\overline d(E)>0$. In particular, $h(n\theta)>0$ for every $n\in E$.
\end{lemma}

\begin{proof}
  For integers $j,N\ge 1$, define the measurable sets
  \[
    W_{j,N}
    :=
    \left\{t\in\mathbb T:h(t)<\frac{2^{-j}}{N}\right\}.
  \]
  Since $h$ vanishes on $C$, we have $C\subseteq W_{j,N}$, which implies
  $m(W_{j,N})\ge m(C)=c$. Next, define
  \[
    \mathcal A_{j,N}
    :=
    \left\{
    \theta\in\mathbb T:
    \bigl|\{1\le n\le N:n\theta\in W_{j,N}\}\bigr|
    \ge
    \frac{c}{2}N
    \right\}.
  \]
  Each set $\mathcal A_{j,N}$ is open, as it is a finite union of finite
  intersections of open sets:
  \[
    \mathcal A_{j,N}
    =
    \bigcup_{
      \substack{F\subseteq\{1,\ldots,N\}\\ |F|\ge \lceil cN/2\rceil}
    }
    \;
    \bigcap_{n\in F}\{\theta\in\mathbb T:n\theta\in W_{j,N}\}.
  \]

  We claim that for any fixed $j,M\ge 1$, the union
  $\bigcup_{N\ge M}\mathcal A_{j,N}$ is dense in $\mathbb T$. Let
  $I\subseteq\mathbb T$ be a non-empty open arc. Applying
  Lemma~\ref{lem:interval-estimate-unified}, for every $N\ge1$ we obtain
  \begin{align*}
    \int_I
    \bigl|\{1\le n\le N:n\theta\in W_{j,N}\}\bigr|
    \,dm(\theta)
    &=
    \sum_{n=1}^N
    \int_I\mathbf 1_{W_{j,N}}(n\theta)\,dm(\theta) \\
    &\ge
    \sum_{n=1}^N
    \left(m(I)m(W_{j,N})-\frac{2}{n}\right) \\
    &\ge
    m(I)cN-2\sum_{n=1}^N\frac1n.
  \end{align*}
  Since $\sum_{n=1}^N 1/n = o(N)$, the final expression strictly exceeds
  $(c/2)m(I)N$ for all sufficiently large $N$. For such an $N\ge M$, the
  set $I\cap\mathcal A_{j,N}$ is non-empty, proving the density claim.

  For each $n\ge1$, the set $\{\theta\in\mathbb T:n\theta\in U\}$ is dense and
  open because multiplication by $n$ is a continuous open surjection on
  $\mathbb T$. Let $\mathbb T_{\mathrm{irr}}$ denote the dense $G_\delta$ set
  of irrational points in $\mathbb T$. By the Baire Category Theorem, the
  intersection
  \[
    \mathbb T_{\mathrm{irr}}
    \cap
    \left(
    \bigcap_{n=1}^{\infty}\{\theta\in\mathbb T:n\theta\in U\}
    \right)
    \cap
    \left(
    \bigcap_{j=1}^{\infty}
    \bigcap_{M=1}^{\infty}
    \bigcup_{N\ge M}\mathcal A_{j,N}
    \right)
  \]
  is a dense $G_\delta$ set. We select a point $\theta$ from this
  intersection. By construction, $\theta$ is irrational, $n\theta\in U$ for
  all $n\ge1$, and for every $j,M\ge 1$ there exists $N\ge M$ such that
  $\theta\in\mathcal A_{j,N}$.

  Set $N_0:=0$. For each $j\ge1$, once $N_{j-1}$ has been chosen, take
  $M=N_{j-1}+1$ and choose an integer $N_j\ge M$ with
  $\theta\in\mathcal A_{j,N_j}$. Then $(N_j)_{j\ge0}$ is strictly increasing.
  Define
  \[
    E_j^+
    :=
    \{1\le n\le N_j:n\theta\in W_{j,N_j}\},
    \qquad
    E_+:=\bigcup_{j=1}^{\infty}E_j^+.
  \]
  Since $\theta\in\mathcal A_{j,N_j}$, we have $|E_j^+|\ge \frac{c}{2}N_j$.
  It follows that
  \[
    \limsup_{N\to\infty}
    \frac{|E_+\cap[1,N]|}{N}
    \ge
    \frac{c}{2}.
  \]
  We define the symmetric set $E:=E_+\cup(-E_+)$. Since $U$ is symmetric,
  $n\theta\in U$ for every $n\in E$. Evaluating the density along the sequence
  $N_j$ yields
  \[
    |E\cap[-N_j,N_j]|
    \ge
    2|E_+\cap[1,N_j]|
    \ge
    cN_j.
  \]
  Consequently, the upper density satisfies $\overline d(E) \ge c/2 > 0$.

  To verify summability, we sum over the sets $E_j^+$. Since
  $|E_j^+|\le N_j$, we have
  \begin{align*}
    \sum_{n\in E_+}h(n\theta)
    &\le
    \sum_{j=1}^{\infty}\sum_{n\in E_j^+}h(n\theta) \\
    &\le
    \sum_{j=1}^{\infty}|E_j^+|\,\frac{2^{-j}}{N_j}
    \le
    \sum_{j=1}^{\infty}2^{-j}=1.
  \end{align*}
  Since $h$ is even,
  $\sum_{n\in E}h(n\theta) \le 2\sum_{n\in E_+}h(n\theta)<\infty$. Finally,
  since $n\theta\in U$ for every $n\in E$, it follows that $h(n\theta)>0$ for
  every $n\in E$.
\end{proof}

\subsection{Proof of Theorem~\ref{thm:positive-density-counterexample}}

Let $\theta$ and $E\subseteq\mathbb Z$ be given by
Lemma~\ref{lem:unified-E}. Let $\Omega:=\{0,1\}^{\mathbb Z}$ and
$\nu:=(\frac{1}{2},\frac{1}{2})^{\mathbb Z}$. Let $S:\Omega\to\Omega$ denote
the two-sided Bernoulli shift given by $(S\omega)_k=\omega_{k+1}$. Define the
product system
\[
  X:=\mathbb T\times\Omega,
  \qquad
  \mu:=m\times\nu,
  \qquad
  T(z,\omega):=(z+\theta,S\omega).
\]
Since $\theta$ is irrational and $S$ is weakly mixing, the system $(X,\mu,T)$
is ergodic. By Lemma~\ref{lem:rotation-bernoulli-pronilfactor}, its
$\infty$-step pro-nilfactor is the circle rotation factor
$\mathcal Z_\infty = \mathcal B_{\mathbb T}\otimes\{\emptyset,\Omega\}$.

Define the measurable subsets $B, A \subseteq X$ by
\[
  B:=\{(z,\omega):z\in O,\ \omega_0=0\}
\]
and
\[
  A:=
  \left\{
  (z,\omega):z\in O,\ \omega_n=1
  \text{ for every } n\in E \text{ such that } z+n\theta\in O
  \right\}.
\]
For each $z\in\mathbb T$, define the set of relevant indices
\[
  F_z:=\{n\in E:z\in O,\ z+n\theta\in O\},
\]
and let $N_E(z):=|F_z| \in\mathbb N\cup\{0,\infty\}$. Equivalently,
\[
  N_E(z)=\sum_{n\in E}\mathbf 1_O(z)\mathbf 1_O(z+n\theta).
\]
By Lemma~\ref{lem:unified-E}, integrating $N_E$ yields
\begin{align*}
  \int_{\mathbb T}N_E(z)\,dm(z)
  &=
  \sum_{n\in E}m(O\cap(O-n\theta)) \\
  &=
  \sum_{n\in E}h(n\theta)<\infty.
\end{align*}
Thus, $N_E(z)<\infty$ for $m$-a.e. $z\in\mathbb T$.

We compute the conditional expectations onto $\mathcal Z_\infty$. For $B$, it
is clear that
\begin{equation}\label{eq:CE-B-unified}
  \mathbb E(\mathbf 1_B\mid\mathcal Z_\infty)(z,\omega)
  =
  \frac{1}{2}\mathbf 1_O(z).
\end{equation}
For $A$, consider the $z$-sections
$A_z:=\{\omega\in\Omega:(z,\omega)\in A\}$. If $z\notin O$, then
$A_z=\emptyset$. If $z\in O$, the condition reduces to
$A_z=\bigcap_{n\in F_z}\{\omega\in\Omega:\omega_n=1\}$. By the independence of
the Bernoulli coordinates, $\nu(A_z)=2^{-|F_z|}$. (If $F_z$ is infinite,
approximating by finite subsets yields $\nu(A_z)=0$.) With the convention
$2^{-\infty}=0$, we have $\nu(A_z)=\mathbf 1_O(z)2^{-N_E(z)}$ for all
$z\in\mathbb T$. This identifies the conditional expectation:
\begin{equation}\label{eq:CE-A-unified}
  \mathbb E(\mathbf 1_A\mid\mathcal Z_\infty)(z,\omega)
  =
  \mathbf 1_O(z)2^{-N_E(z)}.
\end{equation}

Both sets have positive measure: $\mu(B)=\frac{1}{2}m(O)>0$, and
\[
  \mu(A)
  =
  \int_{\mathbb T}\mathbf 1_O(z)2^{-N_E(z)}\,dm(z)
  >
  0,
\]
which is strictly positive because $N_E(z)<\infty$ almost everywhere.

For $n\in\mathbb Z$, define the exact and pro-nilfactor correlations
\[
  \alpha_n
  :=
  \int_X\mathbf 1_A\cdot T^n\mathbf 1_B\,d\mu,
  \qquad
  \beta_n
  :=
  \int_X
  \mathbb E(\mathbf 1_A\mid\mathcal Z_\infty)
  \cdot
  T^n\mathbb E(\mathbf 1_B\mid\mathcal Z_\infty)
  \,d\mu.
\]
We show that the discrepancy set coincides with $E$.

Suppose $n\in E$. If $(z,\omega)\in A\cap T^{-n}B$, then
$T^n(z,\omega)\in B$. This implies $z+n\theta\in O$ and $\omega_n=0$. However,
since $(z,\omega)\in A$, $n\in E$, and $z+n\theta\in O$, the definition of $A$
requires $\omega_n=1$, a contradiction. Hence, $\alpha_n=0$.

Using \eqref{eq:CE-B-unified} and \eqref{eq:CE-A-unified}, we obtain
\[
  \beta_n
  =
  \frac{1}{2}
  \int_{\mathbb T}
  \mathbf 1_O(z)2^{-N_E(z)}\mathbf 1_O(z+n\theta)\,dm(z).
\]
Since $n\in E$, we have $n\theta\in O-O$. Thus, $O\cap(O-n\theta)$ has
strictly positive measure. Because $N_E<\infty$ almost everywhere, the
integrand is strictly positive on a set of positive measure, yielding
$\beta_n>0$. This establishes
$E\subseteq \{n:\beta_n>0\} \setminus \{n:\alpha_n>0\}$.

Conversely, suppose $n\notin E$. Assume $\beta_n>0$. For fixed $z$, the
section $A_z$ depends only on coordinates indexed by $F_z\subseteq E$. Since
$n\notin E$, $\omega_n$ is independent of the coordinates governing $A_z$.
This independence yields
\[
  \nu\bigl(A_z \cap \{\omega:\omega_n=0\}\bigr)
  =
  \frac{1}{2}\nu(A_z)
  =
  \frac{1}{2}\mathbf 1_O(z)2^{-N_E(z)}.
\]
Integrating over $z\in\mathbb T$ yields
\begin{align*}
  \alpha_n
  &=
  \int_{\mathbb T}
  \mathbf 1_O(z)\mathbf 1_O(z+n\theta)
  \nu\bigl(A_z \cap \{\omega:\omega_n=0\}\bigr)
  \,dm(z) \\
  &=
  \frac{1}{2}
  \int_{\mathbb T}
  \mathbf 1_O(z)2^{-N_E(z)}\mathbf 1_O(z+n\theta)\,dm(z)
  =
  \beta_n.
\end{align*}
Thus, $\beta_n>0$ implies $\alpha_n>0$ for every $n\notin E$, meaning no
integer outside $E$ belongs to the discrepancy set. This completes the proof
of Theorem~\ref{thm:positive-density-counterexample}.

\section{Proof of Theorem~\ref{thm:non-piecewise-syndetic}}

We rely on the following combinatorial lemma from \cite{GKLMMRR}.

\begin{lemma}
  {\upshape\cite[Lemma~2.1]{GKLMMRR}}\label{lem:set-algebraic}
  Let $A,B\subseteq\mathbb Z$, and let $\mathcal F$ be an upward-closed family
  of subsets of $\mathbb Z$ such that the intersection of all sets in any
  finite sub-collection of $\mathcal F$ is syndetic. If
  $B\subseteq A\subseteq B-\mathcal F$, then $A\setminus B$ is not piecewise
  syndetic.
\end{lemma}

The main ingredient is a lifting lemma: any finite collection of positivity
conditions on the pro-nilfactor can be realized simultaneously by the exact
correlations along a syndetic set of times.

\begin{lemma}\label{lem:syndetic-lifting}
  Let $(X,\mu,T)$ be an ergodic system, and let $\mathcal Z_\infty$ be its
  $\infty$-step pro-nilfactor. For each $j\in\{1,\ldots,r\}$, let
  $q^{(j)}_1,\ldots,q^{(j)}_{\ell_j}\in\mathcal P_0$ be distinct polynomials,
  and let $h_{j,i}\in L^\infty(\mu)$ be non-negative functions. Suppose that
  \[
    \int_X
    \prod_{i=1}^{\ell_j}
    \mathbb E(h_{j,i}\mid\mathcal Z_\infty)\,d\mu
    >
    0,
    \qquad j\in\{1,\ldots,r\}.
  \]
  Then the intersection set
  \[
    \bigcap_{j=1}^r
    \left\{
    n\in\mathbb Z:
    \int_X
    \prod_{i=1}^{\ell_j}
    T^{q_i^{(j)}(n)}h_{j,i}\,d\mu>0
    \right\}
  \]
  is syndetic.
\end{lemma}

\begin{proof}
  For $j\in\{1,\ldots,r\}$ and $i\in\{1,\ldots,\ell_j\}$, set
  $H_{j,i}:=\mathbb E(h_{j,i}\mid\mathcal Z_\infty)$, and define the sequences
  \[
    \alpha_j(n)
    :=
    \int_X
    \prod_{i=1}^{\ell_j}
    T^{q_i^{(j)}(n)}h_{j,i}\,d\mu,
    \qquad
    \beta_j(n)
    :=
    \int_X
    \prod_{i=1}^{\ell_j}
    T^{q_i^{(j)}(n)}H_{j,i}\,d\mu.
  \]

  We first establish the uniform lower bound
  \[
    \liminf_{N\to\infty}
    \frac{1}{N}\sum_{n=1}^N
    \prod_{j=1}^r \beta_j(n)
    >
    0.
  \]
  Fix $j$. Since $H_{j,i}\ge0$ and the integral of their product is strictly
  positive, there exists a threshold $\eta_j>0$ such that the set
  \[
    C_j:=\bigcap_{i=1}^{\ell_j}\{H_{j,i}>\eta_j\}
  \]
  has positive measure. For $n\in\mathbb Z$, define
  \[
    c_j(n)
    :=
    \mu\left(
    C_j\cap
    \bigcap_{i=1}^{\ell_j}T^{-q_i^{(j)}(n)}C_j
    \right).
  \]
  Consider the product system $X^r$ equipped with the product measure
  $\mu^{\otimes r}$. Let $T_j$ denote the transformation acting as $T$ on the
  $j$-th coordinate and as the identity on all other coordinates. The
  transformations $T_1,\ldots,T_r$ commute. By the polynomial multiple
  recurrence theorem of Bergelson and Leibman~\cite{BL96} applied to the
  product set $C:=C_1\times\cdots\times C_r$,
  \[
    \liminf_{N\to\infty}
    \frac{1}{N}\sum_{n=1}^N
    \mu^{\otimes r}\left(
    C\cap
    \bigcap_{j=1}^r\bigcap_{i=1}^{\ell_j}T_j^{-q_i^{(j)}(n)}C
    \right)
    >
    0.
  \]
  Since the transformations $T_j$ act on distinct coordinates, the measure
  factors completely into $\prod_{j=1}^r c_j(n)$. Hence,
  \[
    \liminf_{N\to\infty}
    \frac{1}{N}\sum_{n=1}^N
    \prod_{j=1}^r c_j(n)
    >
    0.
  \]
  Furthermore, evaluating the pro-nilfactor correlation gives
  \begin{align*}
    \beta_j(n)
    &=
    \int_X
    \prod_{i=1}^{\ell_j}
    T^{q_i^{(j)}(n)} H_{j,i}\,d\mu \\
    &\ge
    \eta_j^{\ell_j}\,
    \mu\left(
    \bigcap_{i=1}^{\ell_j}T^{-q_i^{(j)}(n)}C_j
    \right)
    \ge
    \eta_j^{\ell_j}c_j(n).
  \end{align*}
  Consequently, we obtain a strictly positive limit inferior:
  \[
    L
    :=
    \liminf_{N\to\infty}
    \frac{1}{N}\sum_{n=1}^N
    \prod_{j=1}^r\beta_j(n)
    >
    0.
  \]

  By Theorem~\ref{thm:structural-facts-polynomial-correlations} (2), each
  sequence admits a decomposition $\beta_j(n)=\psi_j(n)+\lambda_j(n)$, where
  $\psi_j$ is a nilsequence and $\lambda_j$ is Banach-null. By
  Lemma~\ref{lem:nilsequence-banach-null-algebra}, their product decomposes as
  $\prod_{j=1}^r\beta_j(n)=\psi(n)+\lambda(n)$, where $\psi$ is a nilsequence
  and $\lambda$ is Banach-null. Since each $\beta_j$ is real-valued, $\psi$
  and $\lambda$ can be taken real-valued.

  By the definition of $L$, for all sufficiently large $N$,
  \[
    \frac{1}{N}\sum_{n=1}^N\prod_{j=1}^r\beta_j(n)>\frac{3L}{4}.
  \]
  Since $\lambda$ is Banach-null,
  $\lim_{N-M\to\infty} \frac{1}{N-M}\sum_{n=M}^{N-1} |\lambda(n)|=0$. Thus,
  for all sufficiently large $N$,
  $\frac{1}{N}\sum_{n=1}^N|\lambda(n)|<\frac{L}{4}$. We deduce that
  \begin{align*}
    \frac{1}{N}\sum_{n=1}^N\psi(n)
    &=
    \frac{1}{N}\sum_{n=1}^N\prod_{j=1}^r\beta_j(n)
    -
    \frac{1}{N}\sum_{n=1}^N\lambda(n) \\
    &\ge
    \frac{1}{N}\sum_{n=1}^N\prod_{j=1}^r\beta_j(n)
    -
    \frac{1}{N}\sum_{n=1}^N|\lambda(n)| \\
    &>
    \frac{L}{2}.
  \end{align*}
  This strict lower bound implies that $\sup_{n\in\mathbb Z}\psi(n)\ge L/2$.
  By Lemma~\ref{lem:nilsequence-positive-set-syndetic}, the set
  $S_0:=\{n\in\mathbb Z:\psi(n)>L/4\}$ is syndetic. Define the exceptional set
  $E_0:=\{n\in\mathbb Z:|\lambda(n)|>L/8\}$. Since $\lambda$ is Banach-null,
  $d^*(E_0)=0$. Therefore, by Lemma~\ref{lem:syndetic-minus-null}, the
  set difference $S_1:=S_0\setminus E_0$ is syndetic. For every $n\in S_1$,
  \[
    \prod_{j=1}^r\beta_j(n)>\frac{L}{8}.
  \]

  For each $j$, let $M_j:=\sup_{n\in\mathbb Z}\beta_j(n)$. Note that $M_j>0$
  for every $j$; otherwise, $\beta_j\equiv 0$, which contradicts $L>0$. Choose
  $\delta>0$ small enough that
  $2\delta\prod_{k\ne j}M_k<\frac{L}{8}$ for all $j\in\{1,\ldots,r\}$ (with
  the convention that an empty product evaluates to $1$). If there were an
  $n\in S_1$ such that $\beta_j(n)\le 2\delta$ for some $j$, it would follow
  that
  \[
    \prod_{m=1}^r\beta_m(n)
    \le
    2\delta\prod_{k\ne j}M_k
    <
    \frac{L}{8},
  \]
  which contradicts the definition of $S_1$. Consequently, $\beta_j(n)>2\delta$
  for all $n\in S_1$ and $j\in\{1,\ldots,r\}$.

  By Theorem~\ref{thm:structural-facts-polynomial-correlations} (1), each
  difference sequence $\alpha_j-\beta_j$ is Banach-null. Defining the error
  sets
  $\mathcal E_j:= \{n\in\mathbb Z:|\alpha_j(n)-\beta_j(n)|\ge\delta\}$, we
  have $d^*(\mathcal E_j)=0$. Their finite union
  $\mathcal E:=\bigcup_{j=1}^r\mathcal E_j$ satisfies $d^*(\mathcal E)=0$.
  By Lemma~\ref{lem:syndetic-minus-null}, the set
  $S_2:=S_1\setminus\mathcal E$ is syndetic. For every $n\in S_2$ and every
  $j$,
  \[
    \alpha_j(n)
    \ge
    \beta_j(n)-|\alpha_j(n)-\beta_j(n)|
    >
    2\delta-\delta
    =
    \delta>0.
  \]
  Thus, $S_2$ is contained in the desired intersection, which is therefore
  syndetic.
\end{proof}

\begin{proof}[Proof of Theorem~\ref{thm:non-piecewise-syndetic}]
  Set $F_i:=\mathbb E(f_i\mid\mathcal Z_\infty)$ for $i\in\{1,\ldots,d\}$, and
  define the sets
  \[
    R_\beta:=\{n\in\mathbb Z:\beta(n)>0\},
    \qquad
    R_\alpha:=\{n\in\mathbb Z:\alpha(n)>0\}.
  \]
  We apply Lemma~\ref{lem:set-algebraic} with $A:=R_\beta$ and $B:=R_\alpha$.

  \medskip
  \noindent
  \emph{Step 1: $R_\alpha\subseteq R_\beta$.}

  Observe that if $f\ge0$ and $F=\mathbb E(f\mid\mathcal Z_\infty)$, then
  $\{f>0\}\subseteq\{F>0\}$ modulo null sets. Indeed, since $F$ is
  $\mathcal Z_\infty$-measurable, the zero set $Z_F:=\{F=0\}$ is
  $\mathcal Z_\infty$-measurable, yielding
  \[
    \int_{Z_F}f\,d\mu
    =
    \int_{Z_F}F\,d\mu
    =
    0.
  \]
  Since $f\ge0$, $f$ vanishes almost everywhere on $Z_F$.

  Let $n\in R_\alpha$. Then
  \[
    \alpha(n)
    =
    \int_X\prod_{i=1}^dT^{p_i(n)}f_i\,d\mu
    >
    0.
  \]
  Since all $f_i$ are non-negative, the intersection
  $G_n:=\bigcap_{i=1}^d\{T^{p_i(n)}f_i>0\}$ has positive measure. Because
  $\mathcal Z_\infty$ is $T$-invariant, conditional expectation commutes with
  the transformation:
  \[
    \mathbb E(T^{p_i(n)}f_i\mid\mathcal Z_\infty)
    =
    T^{p_i(n)}\mathbb E(f_i\mid\mathcal Z_\infty)
    =
    T^{p_i(n)}F_i.
  \]
  Applying the observation above to $T^{p_i(n)}f_i$ gives
  $\{T^{p_i(n)}f_i>0\} \subseteq \{T^{p_i(n)}F_i>0\}$ modulo null sets. Hence,
  \[
    \mu\left(
    \bigcap_{i=1}^d\{T^{p_i(n)}F_i>0\}
    \right)>0.
  \]
  On this set, the product $\prod_{i=1}^dT^{p_i(n)}F_i$ is strictly positive,
  yielding
  \[
    \beta(n)
    =
    \int_X\prod_{i=1}^dT^{p_i(n)}F_i\,d\mu
    >
    0.
  \]
  Therefore, $n\in R_\beta$, proving the inclusion
  $R_\alpha\subseteq R_\beta$.

  \medskip
  \noindent
  \emph{Step 2: Construction of the family $\mathcal F$.}

  Let $\mathcal F$ be the family of all sets $A\subseteq\mathbb Z$ for which
  there exist $\ell\ge1$, distinct polynomials
  $q_1,\ldots,q_\ell\in\mathcal P_0$, and non-negative functions
  $h_1,\ldots,h_\ell\in L^\infty(\mu)$ such that
  \[
    \int_X
    \prod_{i=1}^{\ell}
    \mathbb E(h_i\mid\mathcal Z_\infty)
    \,d\mu
    >
    0
  \]
  and
  \[
    R(q_1,\ldots,q_\ell;h_1,\ldots,h_\ell)
    :=
    \left\{
    n\in\mathbb Z:
    \int_X\prod_{i=1}^{\ell}T^{q_i(n)}h_i\,d\mu>0
    \right\}
    \subseteq A.
  \]
  By definition, $\mathcal F$ is upward-closed.

  We claim that every finite intersection of members of $\mathcal F$ is
  syndetic. Let $F_1,\ldots,F_s\in\mathcal F$. By the definition of
  $\mathcal F$, for each $a\in\{1,\ldots,s\}$ there exists a base set
  $B_a=R(q_1,\ldots,q_\ell;\allowbreak h_1,\ldots,h_\ell)$\allowbreak\
  contained in $F_a$. Each $B_a$ satisfies the corresponding pro-nilfactor
  positivity condition. By Lemma~\ref{lem:syndetic-lifting} applied to this
  finite collection, $\bigcap_{a=1}^s B_a$ is syndetic. Since
  $\bigcap_{a=1}^s B_a \subseteq \bigcap_{a=1}^s F_a$, the intersection
  $\bigcap_{a=1}^s F_a$ is also syndetic.

  \medskip
  \noindent
  \emph{Step 3: $R_\beta\subseteq R_\alpha-\mathcal F$.}

  Fix $t\in R_\beta$. We show that $R_\alpha-t\in\mathcal F$. Define the
  shifted polynomials and functions
  \[
    q_i(n):=p_i(n+t)-p_i(t),
    \qquad
    h_i:=T^{p_i(t)}f_i,
    \qquad i\in\{1,\ldots,d\}.
  \]
  By construction, $q_i(0)=0$, meaning $q_i\in\mathcal P_0$. Furthermore, the
  polynomials $q_1,\ldots,q_d$ remain distinct: for $i\ne i'$, the difference
  $q_i(n)-q_{i'}(n)$ is non-constant since $p_i-p_{i'}$ is non-constant.

  For any $n\in\mathbb Z$, the exact correlation for these shifted parameters
  is
  \begin{align*}
    \int_X\prod_{i=1}^dT^{q_i(n)}h_i\,d\mu
    &=
    \int_X\prod_{i=1}^d T^{p_i(n+t)-p_i(t)}\,T^{p_i(t)}f_i\,d\mu \\
    &=
    \int_X\prod_{i=1}^dT^{p_i(n+t)}f_i\,d\mu \\
    &=
    \alpha(n+t).
  \end{align*}
  Therefore,
  \[
    R_\alpha-t
    =
    \left\{
    n\in\mathbb Z:
    \int_X\prod_{i=1}^dT^{q_i(n)}h_i\,d\mu>0
    \right\}.
  \]
  To verify the positivity condition for $\mathcal F$, we use the
  $T$-invariance of $\mathcal Z_\infty$:
  \[
    \mathbb E(h_i\mid\mathcal Z_\infty)
    =
    \mathbb E(T^{p_i(t)}f_i\mid\mathcal Z_\infty)
    =
    T^{p_i(t)}F_i.
  \]
  Consequently,
  \begin{align*}
    \int_X\prod_{i=1}^d\mathbb E(h_i\mid\mathcal Z_\infty)\,d\mu
    &=
    \int_X\prod_{i=1}^dT^{p_i(t)}F_i\,d\mu \\
    &=
    \beta(t).
  \end{align*}
  Since $t\in R_\beta$, this integral is strictly positive, meaning
  $R_\alpha-t\in\mathcal F$. This establishes the inclusion
  $R_\beta\subseteq R_\alpha-\mathcal F$.

  \medskip
  \noindent
  \emph{Step 4: Conclusion.}

  We have established the inclusions
  \[
    R_\alpha\subseteq R_\beta
    \qquad\text{and}\qquad
    R_\beta\subseteq R_\alpha-\mathcal F,
  \]
  and confirmed that every finite intersection of members of the upward-closed
  family $\mathcal F$ is syndetic. By Lemma~\ref{lem:set-algebraic}, the
  set difference $R_\beta\setminus R_\alpha$ is not piecewise syndetic. This
  set is the zero-threshold discrepancy set.
\end{proof}

\end{document}